\documentclass[a4paper,12pt]{amsart}

\usepackage{amsmath}
\usepackage{amssymb,amsbsy,amsmath,amsfonts,amssymb,amscd}
\usepackage{latexsym}
\usepackage{amsthm}
\usepackage{graphics}
\usepackage{color}
\usepackage{tikz}
 \usepackage{tikz-3dplot}
\usetikzlibrary{arrows}
\usepackage{tikz-cd}
\usepackage{pdfpages} 
\usepackage{longtable}
\usepackage{xy}
\usepackage{array}
\usepackage{color}
\usepackage{comment}

\input xy
\xyoption{all}

\newcommand\sC{{\mathcal C}}

\usepackage{textcomp} 
               
\newcommand\la{\lambda}

\newcommand\be{\beta}

\newcommand\s{\sigma}

\newcommand\de{\delta}

\DeclareMathOperator{\Pic}{Pic}

\def\Bbb{\bf}

\newcommand{\CC}{\ensuremath{\mathbb{C}}}

\newcommand{\ZZ}{\ensuremath{\mathbb{Z}}}
\newcommand{\QQ}{\ensuremath{\mathbb{Q}}}

\newcommand{\PP}{\ensuremath{\mathbb{P}}}

\newcommand{\F}{\ensuremath{\mathbb{F}}}

\def\eea{\end{eqnarray*}}
\def\bea{\begin{eqnarray*}}

\newcommand\dual{\mathrel{\raise3pt\hbox{$\underline{\mathrm{\thinspace d
\thinspace}}$}}}
\newcommand\qe{\ifhmode\unskip\nobreak\fi\quad $\Box$}       

\def\BOX{\hfill\lower.5\baselineskip\hbox{$\Box$}}

\newcommand\Q{\Bbb Q}

\newtheorem{theorem}{Theorem}
\newtheorem{theo}[theorem]{Theorem}
\newtheorem{rem}[theorem]{Remark}

\newtheorem{prop}[theorem]{Proposition}
\newtheorem{cor}[theorem]{Corollary}
\newtheorem{lemma}[theorem]{Lemma}
\newtheorem{example}[theorem]{Example}
\newenvironment{ex}{\begin{example}\rm}{\end{example}}

\newtheorem{claim}[theorem]{Claim}
\newtheorem{fact}[theorem]{Fact}

\newtheorem{main-claim}[theorem]{Main Claim}

\theoremstyle{definition}
\newtheorem{defin}[theorem]{Definition}

\usepackage{hyperref}

\makeatletter
\def\tagform@#1{\maketag@@@{\ignorespaces#1\unskip\@@italiccorr}}
\makeatother

\newcolumntype{H}{@{}>{\lrbox0}l<{\endlrbox}} 

\begin{document}

\title[Singularities of quartic surfaces III]{Singularities  of  normal quartic  surfaces III (char=2, non-supersingular)}
\author{Fabrizio Catanese}
\address{Lehrstuhl Mathematik VIII, 
 Mathematisches Institut der Universit\"{a}t
Bayreuth, NW II\\ Universit\"{a}tsstr. 30,
95447 Bayreuth, Germany \\ and Korea Institute for Advanced Study, Hoegiro 87, Seoul, 
133--722.}
\email{Fabrizio.Catanese@uni-bayreuth.de}

\author{Matthias Sch\"utt}
\address{Institut f\"ur Algebraische Geometrie, Leibniz Universit\"at
  Hannover, Welfengarten 1, 30167 Hannover, Germany\\ and\;\;\;\;\;\;\;\;\;\;\;\;\;\;\;\;\;\;\;\;\;\;\;\;\;\;\;\;\;\;\;\;\;\;\;\;\;
  \linebreak
  Riemann Center for Geometry and Physics, Leibniz Universit\"at
  Hannover, Appelstrasse 2, 30167 Hannover, Germany}

\email{schuett@math.uni-hannover.de}
\date{\today}

\thanks{
The first author acknowledges support of the ERC 2013 Advanced Research Grant - 340258 - TADMICAMT}

\maketitle

\begin{abstract}
We show, in this third part,  that the maximal number of singular
points of a normal quartic surface $X \subset \PP^3_K$ 
defined over an algebraically closed field $K$ of characteristic $2$ is at most   $12$,
if the minimal resolution of $X$ is not a supersingular K3 surface.
We also provide a family of explicit examples,  valid in any characteristic.
\end{abstract}

%
%
%

\section{Introduction}

This paper continues our study of normal quartic surfaces $X \subset \PP^3_K$ 
defined over an algebraically closed field $K$ of characteristic $2$.
This was started in \cite{cat21} and continued in \cite{CS}.
In detail, we proved:

\begin{theo}[{\cite[Thm.\ 1]{CS}}]
\label{maintheo}
\label{theo}
A normal quartic surface $X\subset \PP^3$  
contains at most 14 singular points. 
If the maximum number of 14 singularities is attained,
then all singularities are nodes
and the minimal resolution of $X$ is a supersingular K3 surface.
The variety of quartics with $14$ nodes contains 
an irreducible  component, of dimension $24$.
\end{theo}

Recall that a K3 surface is supersingular if
it has maximum Picard number $\rho=22$ (\cite{artinSS}).
The present paper concerns the non-supersingular case.
In this case, Theorem \ref{theo} implies that $X$ contains at most 13 singular points.
Our main result improves this to the following sharp bound:

\begin{theo}
\label{thm}
Let $S$ be the minimal resolution of a normal quartic surface $X$.
If $S$ is not a supersingular K3 surface, then $X$ contains at most 12 singular points.
If there are 12 singular points, then they all have types $A_1$ or $A_2$,
and there are at most 3 $A_2$'s.
The variety of quartics with $12$ nodes contains 
an irreducible  component, of dimension $22$, such that generically $S$ is a non-supersingular K3 surface.
\end{theo}

Naturally Theorem \ref{thm} leads to the question about what is true for other quasi-polarized K3 surfaces
in characteristic $2$.
We will prove some partial results in this direction, for instance the following.

\begin{theo}
\label{theo2}
Let $X$ be a K3 surface in characteristic $2$
with Picard number $18\leq \rho(X)\leq 20$.
Then $X$ contains at most 12 disjoint smooth rational curves.
\end{theo}

Along the way to prove these results,
we also develop a  subtle characteristic-free relation between disjoint smooth rational curves
and fibre components of elliptic fibrations
which may be of independent interest (Proposition \ref{prop:N_v}).

\medskip

\emph{Convention:}
We work over an algebraically closed field $K$,
mostly of characteristic 2, though many results
may also be stated over non-closed fields.

\section{Elliptic fibrations}
\label{s:g=1}

This section reviews parts of the theory of elliptic fibrations on K3 surfaces
used in the proofs of Theorems \ref{thm} and \ref{theo2}.
Usually, in characteristics $2$ and $3$, this encompasses quasi-elliptic fibrations
as well, but since we presently restrict to non-supersingular K3 surfaces,
this will not be necessary for us
(since quasi-elliptic over $\PP^1$ implies unirational which in turn implies supersingular).

Let $X\subset\PP^3$ be a normal quartic with only rational double points as singularities
(we proved in \cite[Prop.\ 14]{CS} that if $X$ has at least 13 singularities,
then they are all rational double points),
of which we fix $P$ and $Q$.
Denote the minimal resolution of $X$ by $S$.
Then the pencil of hyperplanes containing $P$ and $Q$ 
endows $S$ with a genus one fibration
\[
S \to \PP^1.
\]
All other singular points give disjoint fibre components
(or, rather, disjoint ADE-configurations contained in the fibres of $S$).
For singular points not collinear with $P$ and $Q$, this is obvious;
for the other case, observe that, if the line $L=\overline{PQ}$ contains a third singular point, 
then it is contained in $X$ and it is a multiple component of some plane through $P, Q$,
and its strict transform is contained in the corresponding fibre.

Our interest in normal quartics with many singular points
thus leads us to study elliptic fibrations on K3 surfaces
with many disjoint smooth rational fibre components.
Outside characteristic $2$, there can be as many as 16 disjoint smooth rational fibre components
(realized on  Kummer surfaces isogenous to a product).
In characteristic $2$, however, this is prevented by
the wild ramification at additive fibres.

The wild  ramification $\delta_v$ measures  the discrepancy between the Euler number 
$e(F_v)$ of the fibre $F_v$ and the local multiplicity of the discriminant
(which can be computed on the Jacobian fibration).
The following table, reproduced from \cite{CS}, 
lists standard information on the fibres, given both in terms of Kodaira's types and Dynkin types,
namely the number of irreducible components $m_v$ and the Euler number $e(F_v)$.
It also gives bounds for $\delta_v$ from \cite[Prop.~5.1]{SSc}
and the maximal number $N_v$ of disjoint (-2)-fibre components
which can be inferred directly from inspecting  the corresponding extended Dynkin diagrams.

 \begin{table}[ht!]
 \begin{tabular}{c||c|c|c|c|c|c|c|c|c}
 fibre type & I$_n$ & II & III & IV & I$^*_n \; (n\neq 1)$ & I$^*_1$ & IV$^*$ & III$^*$ & II$^*$\\
 \hline
 \hline
 Dynkin type & $A_{n-1}$ & $A_0$ & $A_1$ & $A_2$ & $D_{n+4}$ & $D_5$ & $E_6$ & $E_7$ & $E_8$\\
 \hline
 $m_v$ & $n$ & 1 & 2 & 3 & $n+5$ & 6 & 7 & 8 & 9\\
 \hline
 $\delta_v$ & 0 & $\geq 2$ & $\geq 1$ & 0 & $\geq 2$ & 1 & 0 & $\geq 1$ & $\geq 1$\\
  \hline
 $e(F_v)$ & $n$ & 2 & 3 & 4 & $n+6$ & 7 & 8 & 9 & 10\\
  \hline
 $N_v$ & $\lfloor \frac{n}{2}\rfloor$ & 0 & 1 & 1 & $4 + \lfloor \frac{n}{2}\rfloor$ & 4 & 4 & 5 & 5\\
 \end{tabular}
 \caption{Singular fibre data in characteristic $2$}
 \label{tab}
 \end{table}

Note that, by inspection of the table,
\begin{eqnarray}
\label{eq:N_v}
N_v \leq \frac 12 (e(F_v) + \delta_{v}).
\end{eqnarray}
Summing over all singular fibres, one obtains the following strong restrictions:

\begin{prop}[{\cite[Prop.\ 22, Cor.\ 24 \& 25]{CS}}]
\label{lem:12}
\label{cor:12fibres}
\begin{enumerate}
\item[(i)]
 In characteristic 2, on an elliptic K3 surface the singular fibres contain at most 12 disjoint $(-2)$-curves.
 \item[(ii)]
 If the fibres of an elliptic K3 surface in characteristic $2$ contain 12 disjoint $(-2)$-curves,
then the only possible singular fibre types are
 (with minimum possible $\delta_v$ each)
 \[
\mathrm I_{2n} \; (n>0), \;\; \mathrm I_{2n}^*\; (n\geq 0), \;\;  \mathrm I_1^*,\;\;   \mathrm{IV}^*,\;\;\mathrm{III}^*.
\]
\item[(iii)]
 If the fibres of an elliptic K3 surface in characteristic $2$ support 12 disjoint ADE-configurations
 of smooth rational curves,
then each has type $A_1$.
\end{enumerate}
\end{prop}


Using explicit calculations with the Weierstrass form of the Jacobian fibration
(which has the same configuration of singular fibres),
we then established the following characterization of
elliptic K3 surfaces whose fibres contain 12 disjoint smooth rational curves:

\begin{prop}[{\cite[Prop.\ 26]{CS}}]
\label{lem:=12}
Let $X$ be an elliptic K3 surface such that there are 
 12 disjoint $(-2)$-curves contained in the fibres.
Then $X$ is supersingular or 
there are two additive fibres.
\end{prop}

Note that the additive fibres have the same types as those in Proposition \ref{lem:12} (ii)
different from I$_{2n}$,
in particular, they are non-reduced.
This will be instrumental for the proof of Theorem \ref{thm}.

\section{Bounding the number of singular points}

Throughout this section we assume
that $X\subset \PP^3$ is a normal quartic containing 13 singular points.
By \cite[Proposition 15]{CS}, all singularities are rational double points,
so the minimal resolution $S$ is a K3 surface
which we assume to be non-supersingular, i.e.\ $\rho(S)\leq 20$.
Note that, in particular, this implies that $S$ and $X$ cannot be unirational
which rules out many cases from \cite{CS} -- especially the quasi-elliptic fibrations.
Hence all genus one fibrations on $S$ are elliptic.
By Theorem \ref{theo} we can restrict to the case where there are exactly 13 singular points.
We first draw some consequences valid for all configurations of singularities.

\subsection{Non-reduced fibre}

As in Section \ref{s:g=1}, consider
an elliptic fibration
\begin{eqnarray}
\label{eq:pi}
S\to\PP^1
\end{eqnarray}
induced by two singular points on $X$.
Recall that this has at least 11 disjoint smooth rational fibre components.

\begin{lemma}
\label{lem:non-reduced}
The fibration \eqref{eq:pi} has a non-reduced fibre.
\end{lemma}

\begin{proof}
If there are at least 12 disjoint smooth rational fibre components,
the statement follows from Proposition \ref{lem:=12}.
Otherwise, 11 disjoint smooth rational fibre components
leave a little more room for fibre types compared to Proposition \ref{cor:12fibres} (ii).
Namely, if we exclude non-reduced fibres,
then Table \ref{tab} and \eqref{eq:N_v} imply that
all fibres have types I$_{2n}$ for varying $n$ (as in Proposition \ref{cor:12fibres} (ii))
except that
\begin{enumerate}
\item[(i)]
either there are 2 I$_n$ fibres with odd $n$ 
\item[(ii-III)]
or there is one fibre of type III (with minimal wild ramification $\delta=1$)
\item[(ii-IV)]
or there is one fibre of type IV.
\end{enumerate}
To rule out all cases, we switch to the Jacobian of \eqref{eq:pi} (with the same singular fibres etc)
and consider its Weierstrass form
\begin{equation}
\label{eq:WF}
y^2 + a_1 xy + a_3 y = x^3 + a_2 x^2 + a_4 x + a_6.
\end{equation}
Here $a_i\in k[t]$, and in order for $S$ (and its Jacobian) to be a K3 surface, one needs $\deg(a_i)\leq 2i$ for each $i$.

A priori, the zeros of $a_1$ give either supersingular fibres (if $a_3\neq 0$ at this point) or additive fibres (if $a_3=0$).
In case (i), there remains only the first alternative  because there cannot be an additive fibre,
and then we can derive a  contradiction  
as follows.
Locate the supersingular fibre at $t=\infty$ and the two multiplicative fibres with odd number of components at $\alpha, \beta$.
By \cite[Lemma 6.4]{RS64-2}, the discriminant $\Delta$ has vanishing $t^{23}$-coefficient,
but spelling this out we get
\[
\Delta = (t+\alpha)^{n_1}(t+\beta)^{n_2} (\hdots)^2 = t^{24} + (\alpha+\beta) t^{23} + \hdots.
\]
By assumption, $\alpha\neq\beta$, so this gives the required contradiction.

In cases (ii-III) and (ii-IV),
we can argue using the fact that the discriminant $\Delta$ is a square, here 
\begin{eqnarray}
\label{eq:Delta}
\Delta = a_3^4+a_1^3a_3^3
+ a_1^4a_4^2
+a_1^4a_2a_3^2
+ a_1^5a_3a_4
+a_1^6a_6.
\end{eqnarray}
This holds since
the vanishing order of $\Delta$ equals the contributions to the Euler number 
at the fibres $I_{2n}$, and these are even; moreover, for type III with $\de =1$  and for type IV, the vanishing order is 4.
Hence all vanishing orders are even, so $\Delta$ is a square.

Since there are no non-reduced fibres by assumption, there is only one additive fibre which we locate at $t=0$.
Arguing exactly as in the proof of Proposition 26 in \cite{CS} (around formula (9))
[which we cited here as Proposition \ref{lem:=12}], 
we can reduce to the case $a_1=t^2$.
Tate's algorithm gives the normal form
 \[
y^2 + t^2 xy + ta_3' y = x^3 + ta_2' x^2 + ta_4'x + t^2 a_6' .
\]
Resolving the singularity shows that we have  type III if and only of $t\nmid a_4'$,
and type IV if and only if $t\mid a_4'$, but $t\nmid a_3'$ (cf.\ \cite[Steps 4 \& 5]{Si3}).

Expand now $\Delta$ as
\[
\Delta = t^4(a_3'^4+t^5a_3'^3+t^6\hdots).
\]
 Thus $\Delta$ could only be square if $t\mid a_3'$
 which is ruled out for type IV and leads to ramification index $\delta_0>2$
 for type III, but this was supposed to be minimal, i.e.\ $\delta_0=1$.
 This gives the required contradiction.

 \end{proof}

\subsection{Nodes as singularities}

We first cover the case where all singular points are nodes.
Since the exceptional curves  above nodes appear with multiplicity one in the fibres, 
we can improve the above results in case $P$ and $Q$ are collinear with another node $R\in X$.

\begin{lemma}
\label{lem:collinear}
If there are 3 collinear nodes,
then the induced fibration has exactly one non-reduced fibre,
and this has type $\mathrm I^*_0$.
The fibre is induced by a plane which intersects $X$ in two double lines,
each containing 3 nodes.
\end{lemma}

\begin{proof}
The 3 collinear nodes are contained in a line $L\subset X$.
The fibres of \eqref{eq:pi} are thus induced by the cubics residual to $L$ in the planes $H$ containing $L$.
Unless $H$ contains $L$ with higher multiplicity,
the residual cubic meets $L$ in 3 different points and is thus reduced.
Since the exceptional curves above the nodes also appear with multiplicity one in the fibres,
the first half of the claim of the lemma follows.
More precisely, using Lemma \ref{lem:non-reduced}, 
there is exactly one non-reduced fibre,
corresponding to the unique plane $H_0$ containing $L$ with higher multiplicity.
Since the nodes on $L$ induce sections
while a non-reduced fibre has at least 5 components,
there have to be 4 components provided by the conic $Q$ residual to $L$ in $H_0$
and by the nodes it contains.
Then $Q$ must be non-reduced,
hence $Q$ is a double line containing 3 nodes as stated,
and we conclude that we have a fibre of type I$^*_0$.
\end{proof}

We take the lemma as the first step to compute all possible non-reduced fibre types for all settings.

\begin{lemma}
\label{lem:PQ}
A fibre of \eqref{eq:pi} is non-reduced if and only if the underlying hyperplane section
is an irreducible double conic (with 6 nodes)
or splits into two double lines (with 3 nodes each).
If the conic is irreducible,
we get a fibre of type  I$_0^*$.
\end{lemma}

\begin{proof}
If the fibration \eqref{eq:pi} is induced by 3 collinear nodes,
then we have seen this already in Lemma \ref{lem:collinear}
(in fact, only the second alternative).

In general, since exceptional curves appear with multiplicity one in the fibres,
a fibre corresponding to some plane $H$ through nodes $P, P'$ can only be non-reduced
if $H\cap X$ is non-reduced.
If it is an irreducible double conic given by $\{q=z=0\}$ where $H=\{z=0\}$, then the equation of $X$,
\begin{eqnarray}
\label{eq:double-conic}
F = q^2 + zg + z^2\hdots,
\end{eqnarray}
directly reveals the 6 nodes given by $\{z=q=g=0\}$ (under the assumption that $X$ is normal and all singular points are nodes).
This verifies the first alternative of the lemma.

Otherwise $H$ splits into a double line $L=\{y=z=0\}\subset X$ and a residual conic given by $\{q=z=0\}$.
As before, the equation
\begin{eqnarray}
\label{eq:double-line}
F = y^2 q + zg + z^2\hdots,
\end{eqnarray}
reveals three nodes on $L$ given by $\{z=y=g=0\}$.
Now consider the elliptic fibration induced by $L$
and conclude by applying Lemma \ref{lem:collinear}.

\end{proof}

Note that the irreducible double conic in Lemma \ref{lem:PQ} only arises if no two of the 6 nodes are collinear with another node
(by B\'ezout's theorem), and that it gives a fibre of type I$^*_0$,
while the other configuration may arise for both set-ups 
(giving type I$^*_0$ and three sections if the two specified nodes are collinear with a third node 
as in Lemma \ref{lem:collinear}, 
resp.\ type  I$^*_1$ and two bisections otherwise).
We emphasize that, as shown in the proof, 
these configurations do not involve a line $\overline{PQ}\subset X$ for two nodes $P, Q\in X$
unless the line contains a third node.

\begin{rem}
The two previous lemmata and their proofs
(esp.\ equations \eqref{eq:double-conic}, \eqref{eq:double-line}) also show
that any non-reduced plane is automatically everywhere non-reduced,
falling into the two alternatives from Lemma \ref{lem:PQ}.
We will thus only refer to non-reduced planes in what follows.
\end{rem}

\subsection{Non-reduced planes}

In this section, we piece together the information about non-reduced fibres
to prove the bound of Theorem \ref{thm} in the case of nodes.
We first limit the possible intersections of non-reduced planes.

\begin{lemma}
\label{lem:non-reduced-planes}
The intersection of two non-reduced planes cannot contain three nodes.
\end{lemma}

\begin{proof}
If the intersection of the non-reduced planes $H_1, H_2$,
the line $L$, say,
were to contain 3 nodes, then $L\subset X$ and  the fibration induced by $L$ would have exactly one non-reduced fibre
 by Lemma \ref{lem:collinear}. This contradicts the fact that both $H_1$ and $H_2$ give non-reduced fibres.
\end{proof}

By Lemma \ref{lem:PQ},
each pair of nodes is contained (together with 4 other nodes) 
in a  non-reduced plane as above.
This implies that any node $P$ is contained in at least 3 distinct   non-reduced planes,
say $H_1, H_2, H_3$.

\begin{claim}
\label{claim}
$H_1$ and $H_2$ meet in exactly two nodes.
\end{claim}

\begin{proof}
Each plane $H_i$ contains $P$ and 5 other nodes.
Since there are 12 nodes other than $P$ in total,
there have to be at least 3 duplicate points.
By Lemma \ref{lem:non-reduced-planes},
this amounts to exactly one duplicate (other than $P$) for each pair $(H_i, H_j)$.
\end{proof}

\begin{claim}
\label{claim2}
Each node is contained in exactly 3 non-reduced planes.
\end{claim}

\begin{proof}
If some node $P$ were contained in 4 non-reduced planes,
then each plane would contain $P$ together with 5 other nodes,
so in total we count
\[
1 + 4\cdot 5 - \binom 42 = 15 \; \text{nodes}
\]
by Claim \ref{claim}, which is absurd.
\end{proof}

\subsection{Census for 13 nodes}
\label{ss:13nodes}

Let $m$ be the number of non-reduced planes.
Each contains 6 nodes by Lemma \ref{lem:PQ}.
On the other hand, each node is contained in exactly 3 non-reduced planes by Claim \ref{claim2}.
Hence
\[
6m = 13\cdot 3.
\]
This shows that a non-supersingular normal quartic cannot contain 13 nodes
(proving a substantial part of Theorem \ref{thm}).

\qed

\subsection{Higher singularities}
\label{ss:higher}

We shall now assume that one of the 13 singular points is not a node
(but of ADE-type by \cite[Prop.\ 15]{CS}).
We will derive a contradiction in 3 steps:
\begin{enumerate}
\item[1.]
limit the possible configurations of singularities (Lemma \ref{lem:higher});
\item[2.]
prove that certain fibrations admit suitable sections  (Lemma \ref{lem:section});
\item[3.]
play this off against non-reduced fibres and the restrictions 
on disjoint smooth rational fibre component (Section \ref{ss:vs}).
\end{enumerate}

\begin{lemma}
\label{lem:higher}  
The only options for non-nodes are 
\begin{enumerate}
\item[(1)]
one $A_3$-singularity and 12 nodes, or
\item[(2)]
at most three $A_2$-singularity and all other singularities nodes.
\end{enumerate}
\end{lemma}

\begin{rem}
In Proposition \ref{prop:A_3} we will prove by a different argument
that with 12 or more singular points, a non-supersingular quartic can only admit $A_1$ and $A_2$ singularities.
More precisely, by Proposition \ref{prop:A_2} there can be at most 3 $A_2$'s.
\end{rem}

\begin{proof}
Assume that there is a singular point $P$ which is not of type $A_1$ or $A_2$.
Then $A_3$ embeds into the corresponding Dynkin diagram.
If $P$ lies in a fibre of some elliptic fibration induced by two or three other singularities, 
then the exceptional curves furnish it
with 2  disjoint $(-2)$-curves
(corresponding to the embedding $A_1^2\hookrightarrow A_3$) -- 
also disjoint to the 10 exceptional curves above  the other singular points on the fibres
(and the connecting line if there are three collinear base points).
By Proposition \ref{lem:12} (iii), the corresponding 12 orthogonal root lattices of type $A_1$
supported on the fibres
cannot be extended while staying disjoint.
Hence $P$ has type $A_3$ and the other nine or ten singular points on the fibres  are nodes.
By symmetry, this implies that all singular points other than $P$ are nodes as stated in (1).

It remains to bound the number of $A_2$-singularities.
For this purpose, we introduce the following notation for a fibre $F_v$:
\[
N_v^{(i)} = \max \left\{\begin{matrix} r; \; \exists \, \text{disjoint ADE-configurations $\mathcal C_1,\hdots,\mathcal C_r$ supported on $F_v$}\\ 
\text{such that $i$ of the $\mathcal C_j$'s contain $A_2$'s
}\end{matrix}\right\}.
\]
One might expect that all other ADE-configurations in the above set-up will have type $A_1$,
but then one notices  that except for type III,
on all fibre types not attaining equality in \eqref{eq:N_v} for the minimal $\delta_v$,
one can replace one $A_1$ by $A_2$ while preserving orthogonality.
In fact, this is the first step towards proving the following:

\begin{fact}
\label{fact}
$N_v^{(i)} \leq  \frac 12 (e(F_v) + \delta_{v} - i)$ unless $i=3$ and $F_v$ has type $\mathrm{IV}^*$ where $N_v^{(3)}=3$.
\end{fact}

\begin{proof}
This can be verified directly by going through all Kodaira types.
Note that the bound is sharp for small $i$,
but ceases to stay sharp for larger $i$.
\end{proof}

Seeking for a contradiction,
we assume that $X$ contains 13 singular points, among them four $A_2$ singularities.
Pick an elliptic fibration given by two or three nodes
such that 4 $A_2$'s contribute to the fibres - together with 7 disjoint $A_1$'s (or higher singularities).
If there is a fibre of type IV$^*$ supporting three $A_2$'s, say at $v=\infty$,
then there is one other fibre containing an $A_2$.
For the number of disjoint ADE-configurations including the 4 $A_2$'s supported on the fibres, we thus obtain,
by Fact \ref{fact},
\[
11 \leq 3 +  \sum_{v\neq \infty} N_v^{(i_v)} \leq 3 +  \frac 12 \sum_{v\neq\infty} (e(F_v) + \delta_{v}) -\frac 12 \leq 3 + \frac{16-1}2,
\]
yielding the desired contradiction. If there are no 3 $A_2$'s supported on a fibre of type IV$^*$,
then we can even rule out 3 $A_2$ singularities on the fibres since, by Fact \ref{fact},
\[
11 \leq  \sum_v N_v^{(i_v)} \leq \frac 12 \sum_v (e(F_v) + \delta_{v} - i_v) \leq \frac{24-3}2
\]
gives again a contradiction.

\end{proof}

\subsection{Fibrations with sections}

\begin{lemma}
\label{lem:section}
For every fibration induced by a singular point $P$ which is not a node, and by one or two nodes,
there is a section orthogonal to 11 disjoint $(-2)$-curves supported on the fibres.
\end{lemma}

\begin{proof}
We consider first the case where the fibration is induced by $P$ and a node $P'$ 
which are not collinear with a third singular point.

Consider the configuration $\sC$ of $(-2)$-curves given by the exceptional curves lying above $P$ and $P'$,
 plus the strict transform of the connecting line
in case it is contained in $X$.
The configuration of these $(-2)$-curves
is a priori among the types $A_1+A_2, A_1+A_3, A_4, A_5, D_5$.

Except  for the $D_5$ case, 
the fibration $|F|$ is obtained 
 from the pull-back of the hyperplane divisor $H$
by subtracting the
 fundamental cycle(s) of the configuration $\sC$ (which is reduced).

Then one of the two outer components of the exceptional divisor above $P$ is an outer component of the fundamental
cycle: hence it induces the claimed section
(since it intersects $F$ with multiplicity 1).

In the $D_5$ case, in order 
to obtain the fibration  one has  to subtract all the curves in the configuration $\sC$
with multiplicity one,
except that the exceptional curve  $E$ which has three adjacent curves in the configuration has to be subtracted
with multiplicity 2. In fact, the line $L : = \overline{PP'}$ must be contained inside $X$, and since its strict transform intersects $E$, 
this line $L$  is the intersection of the two planes which form the tangent cone at $P$, which is a singularity of type $A_3$,
$ xy = z^4$:
then each plane containing $L$ contains the sum of the exceptional divisor over $P$ with $E$, as a local calculation
shows.

 In particular,  $E$  provides  a section.

In all  cases, the section is disjoint from the remaining 11 singular points, so the claim follows.

If $P$ is collinear with nodes $P', P''$,
then the exceptional curves and the strict transform of the line (which meets three other components) may form the configurations $D_5, D_6$ or $\tilde D_5$.
The last case is excluded by Proposition \ref{lem:12} (i), 
since then we have a Kodaira fibre of type I$_1^*$ (of some elliptic fibration)
which contains 4 disjoint $(-2)$-curves and is disjoint from the 
remaining 10 singularities.
In the first two cases, the fibration $|F|$ is obtained again
by subtracting the reduced divisor supported on the configuration $\sC$.
All exterior components of the configuration provide sections while the interior components give fibre components
(since they have zero intersection with $F$).
In particular, each of these sections is disjoint from 11 $(-2)$-curves supported on the fibres,
given by one interior component of the configuration and the exceptional curves above the 10 remaining singular points.

\end{proof}

\subsection{Sections vs.\ non-reduced fibres}
\label{ss:vs}

To complete the argument for higher singularities, 
consider a fibration induced by $P$ and some nodes as in Lemma \ref{lem:section}.
By Lemma \ref{lem:non-reduced} there is a non-reduced fibre $F_v$, say with extended Dynkin diagram $\tilde V$.
By Lemma \ref{lem:section}, there is a section $O$ and the disjoint ADE-configuration supported on $\tilde V$ 
is already supported on the Dynkin diagram of type $V$
obtained by omitting the simple fibre component met by $O$.
Denote by $N_v'$ the maximal number of disjoint $(-2)$-curves supported on $V$.
A case-by-case inspection teaches us that this is one less than $N_v$:

\begin{fact}
\label{fact2}
For a non-reduced fibre, one has
$$
N_v' \leq \frac 12 (e(F_v) + \delta_{v})-1.
$$
\end{fact}

\begin{rem}
The same inequality holds true when we omit any other odd multiplicity  component 
of a non-reduced fibre.
\end{rem}

For the fibres to contain 11 disjoint $(-2)$-curves, 
this implies that all inequalities \eqref{eq:N_v} are in fact inequalities,
and the classification of possible fibre types in Proposition \ref{lem:12} (ii) is still valid. 
In particular, this implies that the fibration supports indeed 12 disjoint $(-2)$-curves.
(This can also be seen by adding the component met by $O$, and maybe moving one component to a simple component,
or shortening an $A_2$ configuration to $A_1$ disjoint from the  component added.)
Since $S$ is not supersingular by assumption,
Proposition \ref{lem:=12} and the ensuing remark imply that there are in fact two non-reduced fibres.
Hence Fact \ref{fact2} applies to both of them, 
and there can be at most 10 disjoint $(-2)$-curves which are supported on the fibres 
while being orthogonal to $O$. This contradiction completes the proof that
a normal quartic whose resolution is not supersingular cannot contain 13 (or more) singular points.

\subsection{A 22-dimensional family of non-supersingular quartics with 12 nodes}

\begin{ex}
\label{ex:12}
To exhibit a quartic with 12 nodes over a field $k$ of arbitrary characteristic, in fact),
consider 4 linear forms $l_1,\hdots,l_4$ and a quadric $q\in k[x_1,x_2,x_3,x_4]$.
Then the quartic 
\begin{eqnarray}
\label{eq:family}
l_1\cdots l_4 + q^2 = 0
\end{eqnarray}
has generally 12 nodes.
To see that it is generally not supersingular,
it suffices to specialize to the case where $l_i=x_i, q=\lambda(x_1+\hdots+x_4)^2$.

Then the pencil of quartics is the classical Dwork pencil, where each quartic has  $6$ $A_3$-singularities,
(over $\CC$, it would be the mirror of the Dwork pencil, cf. \cite{ES-family}),
and at $\lambda=-1/81$  (\cite{ES-family}, section 11) we get   a K3 surface $S$ with $\rho=20$, 
both over $\QQ$ (or $\CC$) and in characteristic $2$.
To see this, we use that $S$ contains the twisted cubic $C$ parametrized by
\[
\PP^1\ni t \mapsto (-t^3 
,1-t,t,(t-1)^3).
\]
Together with the exceptional curves and the 4 obvious lines in the plane $\{x_1+\hdots+x_4=0\}$,
the twisted cubic generates $\Pic(S_\QQ) = \Pic(S_\CC)$,
a hyperbolic lattice of rank 20 and determinant $-7$.
It follows that the Picard number of the reduction $S_p=S\otimes\bar\F_p$ at a prime $p\neq 7$ is controlled
by the field $\QQ(\sqrt{-7})$ as follows:
\[
\rho(S_p) = 
\begin{cases}
20 & \text{ if $p$ splits in } \QQ(\sqrt{-7});\\
22 & \text{ if $p$ is inert in } \QQ(\sqrt{-7})
\end{cases}
\]
(cf.\ \cite[ Rem. 13]{S-Pic20}  for instance).
Since $2$ splits in $\QQ(\sqrt{-7})$, we infer that $\rho(S_2)=20$ as claimed.
\end{ex}

%

\begin{rem}
In \cite[Thm.\ 14]{cat21}, there is given a 3-dimensional family of $\mathfrak S_4$-invariant quartics
with 12 nodes, of the form 
$$ a_1 \s_1(x)^4 + a_2  \s_1(x)^2 \s_2(x) + a_3 \s_1(x) \s_3(x) + a_4 \s_4(x) + \be \s_2(x)^2 =0 ,$$
where $ a_2 a_4 = a_3^2$.

Setting $a_2= a_3=0$, we get  two-dimensional subfamily
of \eqref{eq:family} with an $\mathfrak S_4$-symmetry (that is, the linear forms $l_i$ 
become the coordinates and we let $q$ to be a symmetric polynomial).

Using the  one-dimensional subgroup of PGL$(4)$ preserving the symmetry
($x_i\mapsto \lambda x_i+\mu\sigma_1(x)$),
we see that  the surfaces in the 3-dimensional family are projectively equivalent
to the surfaces in the subfamily $a_2=a_3=0$; indeed this follows by \cite[Proposition\ 15]{cat21}
since any point of the form $(1,1,b,c)$ is transformed to a point of
the form $(0,0,1,b')$ by such a transformation with $\la + \mu (b +c)=0$.
\end{rem}

\begin{rem}
Using deformation theory, for instance starting from the above 
two-dimensional family
with an $\mathfrak S_4$-symmetry,
one can show that there are non-supersingular quartics in $\PP^3$
with any number of nodes up to 12 (in any characteristic).
Indeed, the semiuniversal deformation of nodes of equation
\[
z^2 + uv 
\]
is 
given by
\[
z^2 + uv  + a + b z .
\]
We emphasize the following dichotomy:
while $z^2 + uv  + a$  yields a smoothing in characteristic different from $2$ for $a\neq 0$, 
this is equisingular in characteristic $2$;
in turn,
$
z^2 + uv  + b z$
is equisingular in characteristic different from $2$, but yields a smoothing in characteristic $2$ for $b\neq 0$.
In either case, this can be made to work at the 12 nodes of the above shape to prove the claim.
\end{rem}

%
%
%
%
%
%
%
%

\section{Generalizations}
\label{s:general}

Proposition \ref{lem:12}
applies to any fibre class $F$ of a genus one fibration
perpendicular to more than 12 $(-2)$-curves.
Here we will extend this result to any non-trivial isotropic vector $E$ (replacing $F$);
the argument is surprisingly subtle
-- despite claims to the contrary, see the discussion in \cite{RS}, especially Remark 2.10.

A key ingredient is provided by the following result  from \cite{S-nodal}
about divisibilities among $(-2)$-curves
which, over $\CC$, can be proved  (and was already known before for some cases) by topological methods.

\begin{theo}
\label{thm:nodal}
Let $R\subset \Pic(X)$ be a root lattice generated by $(-2)$-curves on a K3 surface $X$.
Denote the primitive closure by
\[
R'=(R\otimes\QQ)\cap \Pic(X) \;\;\; 
\text{ and let } \;\;\;
D\in R'\setminus R.
\]
Then $D$ is neither effective nor anti-effective.
In particular, $D^2\leq -4$.
\end{theo}

Let us now state the generalization of Proposition \ref{lem:12} (i), the main characteristic 2 result for this section.

\begin{prop}
\label{prop:general}
Let $X$ be a K3 surface in characteristic $2$,
endowed with a non-zero isotropic vector $E\in\Pic(X)$ and
at least 13 disjoint $(-2)$-curves orthogonal to $E$.
Then the genus one fibration induced by $E$ is quasi-elliptic.
\end{prop}

\begin{proof}
By standard arguments (see e.g.\ \cite[\S 5]{CS}), we may assume that $E$ is effective and primitive,
and we can apply a composition of reflections $\sigma$ in $(-2)$-curves
such that $E'=\sigma(E)$ has no base locus, i.e. it is the fibre class $F$ of a genus one fibration on $X$.

Denote the $(-2)$-curves perpendicular to $E$ by $C_1,\hdots,C_s \, (s>12)$.
The reflections map these curves to $(-2)$-divisors 
\[
B_i = \sigma(C_i) \perp F
\]
which are effective or anti-effective by Riemann--Roch and thus supported on the fibres of $|F|$
-- or more precisely on a single fibre $F_v$ each.
Proposition \ref{prop:general}
will follow at once as soon as we know that
the number of mutually orthogonal $(-2)$-divisors $B_i$ obtained as above and supported on a single fibre $F_v$ 
 of an elliptic fibration 
still satisfies \ref{eq:N_v}.
This follows in greater generality from the next instrumental proposition.
\end{proof}

The next auxiliary result arises naturally from the problem of embedding $(-2)$-curves 
into reducible fibres via reflections. 
We emphasize that it does not depend on the characteristic.

\begin{prop}
\label{lemma:N_v}
\label{prop:N_v}
Let $X$ be a K3 surface  in arbitrary characteristic 
 endowed with a genus one fibration.
Assume that there are mutually orthogonal $(-2)$-divisors $B_1, \hdots, B_r$ supported on a single fibre $F_v$
which have arisen from disjoint $(-2)$-curves by way of reflections as above.
Then $r\leq N_v$ for $N_v$ from Table \ref{tab} which we reproduce here for the convenience of the reader.
\end{prop}

 \begin{table}[ht!]
 \begin{tabular}{c||c|c|c|c|c|c|c|c}
 fibre type of $F_v$ & I$_n \, (n\geq 0)$ & II & III & IV & I$^*_n \, (n\geq 0)$  & IV$^*$ & III$^*$ & II$^*$\\
 \hline
 $N_v$ & $\lfloor n/2\rfloor$ & $ 0 $ & $1$ & 1 & $\lfloor 4+n/2\rfloor$  & 4 & $5$ & $5$
 \end{tabular}
 \end{table}
 
 \begin{rem}
 The entries of the table are compatible with the bounds in characteristic $2$  in \ref{eq:N_v}
which are due to the wild ramification.
 \end{rem}

 \begin{rem}
The statement of Proposition \ref{lemma:N_v} is not valid without the assumption that the $B_i$'s  arise from disjoint $(-2)$-curves
by way of reflections. In fact, each of the lattices  $D_{2n} \, (n\geq 2)$ admits $A_1^{2n}$ as a finite index sublattice
(as we shall exploit in \ref{sss:D_2m}), so $A_1^{2n}$ embeds into the corresponding fibre,
but for $n>2$  this is not compatible with the assumptions in the proposition. (For $n=2$, compare Example \ref{ex}.)
 \end{rem}

\begin{proof}
Assume that the $(-2)$-divisors $B_1,\hdots, B_r$ are supported on a single fibre $F_v$.
In particular, each $B_i$ embeds into the negative-semi-definite root lattice of $F_v$.
Here classical arguments (cf.\ e.g.\ \cite{Nishi}) 
imply that the embedding always is unique up to reflections
in fibre components -- and up to fibre multiples (part of the  subtlety announced, see Example \ref{ex}).
Hence we will continue to apply the appropriate reflections 
without introducing new notation for $\sigma$ and for the $B_i$'s.

\subsection{Type $\tilde A_n$}
We start with Dynkin type $\tilde A_n \, (n>0)$ (corresponding to Kodaira type I$_{n+1}$).
If $n=1$, we're done. Else we get (e.g.\ by \cite[Cor.\ 4.4]{Nishi})
\[
B_1^\perp = \ZZ F \oplus L \;\;\; \text{ with } \;\;\; L_{\mbox{root}} = A_{n-2}.
\]
Inductively we can thus read off that $\tilde A_n$ supports no more than $(n+1)/2$ 
orthogonal $(-2)$-divisors, confirming the claim.

\subsection{Additive fibre types}
Turning to the remaining  Dynkin types $ V =  D_n,  E_n$,
the fibre corresponds to the extended Dynkin type $\tilde V = \tilde D_n, \tilde E_n$.
We fix an embedding 
$$
\iota: \;\; V\hookrightarrow \tilde V$$
(usually obtained by omitting any fixed simple fibre component of $F$).
In the converse direction,
there is a surjection
\[
\pi: \tilde V \to V
\]
obtained by considering divisors modulo $F$,
which is a left inverse for $\iota$.
In particular, $\pi$ is injective on the 
$\ZZ$-span $\tilde M=\langle B_1,\hdots,B_r\rangle_\ZZ\subset \tilde V$ 
of the orthogonal 
$(-2)$-classes $B_1,\hdots, B_r$:
\[
\pi|_{\tilde M}: \tilde M \hookrightarrow V.
\]
Then we can uniquely (up to the given choice of fibre component) identify 
\[
v\in\pi(\tilde M)=:M \;\;\; \text{ with a vector} \;\;\; 
v+ m_v F \in \tilde M \,\;\; (m_v\in\ZZ).
\]
It is this fibre multiple which makes things quite subtle as illustrated by the next example.

\begin{ex}
\label{ex}
For fibre type $\tilde D_4$, there is nothing to prove in Proposition \ref{prop:N_v}
as there are
obviously  four $(-2)$-curves disjoint embeddings into such a fibre (and no more). 
Moreover, the embedding
\[
\tilde M = 4 A_1 \hookrightarrow \tilde D_4
\]
is primitive, but this ceases to hold true  modulo the fibre.
Indeed, then $M=4A_1$ has index two inside $D_4$,
and this is explained by the fact that
 the sum of the curves and the fibre becomes 2-divisible.
\end{ex}

Note that, if Proposition \ref{prop:N_v} were not to hold for Dynkin type $\tilde D_5$,
then the above arguments would give an embedding $A_1^5\hookrightarrow D_5$.
Since the square classes of the discriminants of the two lattices do not agree, this is impossible.
Hence we will assume that $n\geq 6$ in what follows.

\subsection{Proof strategy}

Consider $\tilde V = \tilde D_n, \tilde E_n\; (n\geq 6)$
and assume that the $(-2)$-divisors 
$B_1,\hdots, B_r$ embed orthogonally into $\tilde V$, with $r$ exceeding the bound given in Proposition \ref{prop:N_v} by one.
We shall now consider 
\begin{itemize}
\item
the lattice $M=\langle\pi(B_1),\hdots,\pi(B_r)\rangle_\ZZ\cong A_1^r\subset V$,
\item
its primitive closure $M'=(M\otimes\QQ)\cap V$, and 
\item
its orthogonal complement $M^\perp\subset V$.
\end{itemize}
We  proceed by comparing their $2$-lengths $l_2(M)$, i.e.\ the lengths (defined as the minimum number of generators)
of the 2-parts of their discriminant groups $A_M = M^\vee/M$ etc
(using the theory laid out by Nikulin \cite{Nikulin}).
The proof has 3 steps:
\begin{enumerate}
\item[1.]
prove that $[M':M]\geq 4$ (Lemma \ref{lem:index});
\item[2.]
show that $M'/M$ contains a subgroup of size 4 
whose elements are represented by roots (Lemma \ref{lem:roots});
\item[3.]
derive a contradiction using Theorem \ref{thm:nodal}.
\end{enumerate}

\subsection{Index $[M':M]$}

\begin{lemma}
\label{lem:index}
In the above set-up, we have the index $[M':M]\geq 4$.
\end{lemma}

\begin{proof}
The next table collects all data relevant to prove the lemma.
The first 2 rows are immediate;
for the third row, compare \cite[Table 2.4, p.\ 33]{MWL}.
For the 4th row, we use that, with $M'\oplus M^\perp$ embedding into $V$
(with summands embedding primitively by definition),
the discriminant groups of $M'$ and $M^\perp$ have to be compatible,
i.e.\ each contains the cokernel $V/(M'\oplus M^\perp)$ as a subgroup.
In practice, this can be quite complicated,
but at any rate it gives the bound
\begin{eqnarray}
\label{eq:l_2}
l_2(M') \leq l_2(V) + l_2(M^\perp),
\end{eqnarray}
 as we are  now going to show.

Indeed, write $N : = M'$, so that $N ,  N'  : = N^\perp = M^\perp $ are primitively embedded into $V$,
that is, we can write $V= N \oplus T = T'  \oplus N'$ which is a direct sum, but not orthogonal.

The roles of $N, N'$ are symmetric, and we observe that we have a series of inclusions 
$$ 0 \subset N \oplus N' \subset V \subset V^{\vee} \subset N^{\vee} \oplus (N')^{\vee}$$
where some  inclusions  are given by the quadratic form.
Let us first take the quotient by $N$: we get then
$$ 0 \subset  N' \subset T \subset N^{\vee}/ N \oplus T^{\vee} \subset N^{\vee}/ N \oplus (N')^{\vee},$$
then we take the quotient by $N'$, yielding
$$ 0 \subset  T/ N' \subset N^{\vee}/ N \oplus T^{\vee}/N'  \subset N^{\vee}/ N \oplus (N')^{\vee}/ N'.$$

To obtain a system of generators of the discriminant group $N^{\vee} / N$ 
it suffices to lift a system of generators of  $V^{\vee} / V$, the quotient of the second by the first piece of the filtration, 
 and then  to lift a system of generators of  $(N')^{\vee} / N'$, 
 which contains $T/ N'$: then the submodule generated by above elements
 surjects onto $(N)^{\vee} / N$.
 The same argument applies for the generators of the binary parts.

Since the length is trivially bounded by the rank, i.e.\ $l_2(M^\perp)\leq \mathrm{rk} M^\perp = n-r$,
adding the entries in the second and third rows gives the bounds in the fourth row.

 \begin{table}[ht!]
 $$
 \begin{array}{c||c|c|c|c|c}
 \text{Dynkin type }  V&  D_{2m} \, (m>2) &  D_{2m+1} \, (m>2)  & E_6 & E_7 & E_8\\
 \hline
 \hline
 r = \mathrm{rk} M & m+3 & m+3 & 5 & 6 & 6\\
 n - r = \mathrm{rk} M^\perp & m-3 & m-2 & 1 & 1 & 2\\
 l_2(V) & 2 & 1 & 0 & 1 & 0\\
 l_2(M') \leq & m-1 & m-1 & 1 & 2 & 2\\
 \end{array}
 $$
 \end{table}
 
 Finally we compare the 2-lengths of $M$ and $M'$.
 Letting $\mu = l_2(M) - l_2(M')$, we infer that
 $2^\mu\det(M')\mid\det(M)$, 
 so standard formulas yields
 \[
 [M':M] \geq 2^{\lceil\frac\mu 2\rceil}.
 \]
 Since $l_2(M)=r$, we get $\mu\geq 4$ for each Dynkin type. 
 This  implies the lemma.
\end{proof}

\subsection{Roots in $M'\setminus M$}

\begin{lemma}
\label{lem:roots}
There is a subgroup $H\subset M'/M$  of size $4$ 
whose elements are represented by roots.
\end{lemma}

\begin{proof}
The overlattice $M'$ is encoded in some isotropic subgroup of $A_M$
(of size at least $4$ by Lemma \ref{lem:index}).
Since $M\cong A_1^r$, 
we have $A_M\cong(\ZZ/2\ZZ)^r$, and all isotropic vectors in $A_M$ are represented by roots
if $r<8$
(precisely $(a_1+\hdots+a_4)/2$ and its permutations). 
This settles the lemma for $n\leq 9$
and leaves types $\tilde D_n$ for $n\geq 10$.
We shall use the standard fact (cf.\ e.g.\ \cite[Cor.\ 4.4]{Nishi}) that 
\begin{eqnarray}
\label{eq:D_n}
A_1^\perp \cong A_1 \oplus  D_{n-2}\subset D_{n} \;\;\; \forall \, n\geq 4
\end{eqnarray}
(with the  convention that $D_2=A_1^2$  and $D_3=A_3$).

Geometrically, we can see this from the diagram

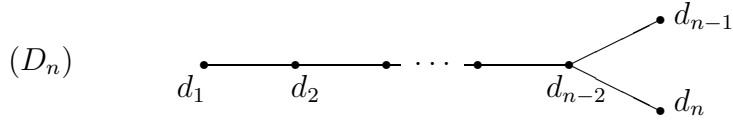
\begin{figure}[ht!]
\setlength{\unitlength}{.6mm}
\begin{picture}(100,25)(5,-5)
\put(-40,7){$(D_n)$}
\multiput(23,8)(20,0){4}{\circle*{1.5}}
\put(23,8){\line(1,0){23}}
\put(83,8){\line(-1,0){23}}
\put(49,8){$\hdots$}
\put(83,8){\line(2,1){20}}
\put(83,8){\line(2,-1){20}}
\put(103,18){\circle*{1.5}}
\put(103, -2){\circle*{1.5}}

\put(23,8){\line(-1,0){20}}
\put(3,8){\circle*{1.5}}

\put(-3,1){$d_1$}
\put(22,1){$d_2$}
\put(78,1){$d_{n-2}$}
\put(106,17){$d_{n-1}$}
\put(106, -2){$d_{n}$}

\end{picture}
\end{figure}

by taking the original $A_1$ to be generated by $d_1$.
Then its orthogonal complement visibly contains $ \langle d_3,\hdots,d_n\rangle\cong D_{n-2}$,
but also another orthogonal summand generated by the fundamental cycle
$\gamma_n = d_1+2(d_2+\hdots+d_{n-2})+d_{n-1}+d_n$.

\subsubsection{Type $\tilde D_{2m} \; (m>4)$}
\label{sss:D_2m}

Applied successively to each summand of $M=A_1^r \; (r=m+3)$ embedding into $D_{2m}$, 
equation \eqref{eq:D_n}
implies  
that $M^\perp$ is an overlattice of $A_1^{m-3}$ 
(since each $D_{2k}$ is an overlattice of $A_1^{2k}$).
Apply now \eqref{eq:D_n} to the finite index sublattice $A_1^{m-3}\subset M^\perp$
to find the finite index inclusions
\[
M' = (M^\perp)^\perp \supseteq D_6\oplus A_1^{m-3} \supset M.
\]
The rightmost inclusion identifies the desired subgroup  $H\subset M'/M$
as $H \cong D_6/A_1^6\cong (\ZZ/2\ZZ)^2$
whose elements are represented by roots
(by the same argument used to reduce to $n\geq 10$).

\subsubsection{Type $\tilde D_{2m+1}\; (m>4)$}

This case is reduced to the previous one by the following general lemma:

\begin{lemma}
For all $r, m>0$, any embedding $A_1^r\hookrightarrow D_{2m+1}$ factors through $D_{2m}$
(primitively embedded in $D_{2m+1}$).
\end{lemma}

\begin{proof}
{
The statement about the primitive embedding is obvious
since otherwise the primitive closure $D'$ of $D_{2m}$ inside $D_{2m+1}$ would be a unimodular lattice
(by inspection of its discriminant group),
its orthogonal complement would be a rank one summand $T$,
and the sum $D'\oplus T$ would be equal to $D_{2m+1}$, since
the quotient $D_{2m+1}/(D'\oplus T)$ embeds into the discriminant group $A_{D'}$ which is trivial.
For determinant reasons,
we have $T\cong \langle-4\rangle$ contradicting the fact that $D_{2m+1}$ is generated by roots.
}

We 
{continue to prove  the
 main statement of} the lemma by induction on $m$.

For $m=1$, we have, by our convention, $D_3=A_3$ which only allows for $r=1,2$.
That is, $A_1$ and $A_1^2= D_2$ embed into $D_3=A_3$,
and the claim is already there.

For the induction step from $m-1$ to $m$,
we first embed one copy of $A_1$ into $D_{2m+1}$.
By \eqref{eq:D_n}, the remaining copies $A_1^{r-1}$ embed as follows:
\begin{eqnarray}
\label{eq:embed}
A_1^{r-1}
\hookrightarrow
A_1 \oplus D_{2m-1}.
\end{eqnarray}
There are two cases.
If the embedding \eqref{eq:embed} involves the first orthogonal summand of the target lattice,
then 
$$A_1^{r-2}\hookrightarrow D_{2m-1} \;\; \Longrightarrow  \;\;
A_1^{r-2}\hookrightarrow D_{2m-2} \hookrightarrow D_{2m-1}
$$
by the induction hypothesis.
 By \cite{Nishi}[Lemma 4.2 (ii)], $D_{2m-2}$ embeds uniquely, up to isometries, as $\langle d_2,\hdots, d_{2m-1}\rangle$
in the notation of the previous figure.
Hence 
\begin{eqnarray}
\label{eq:case1}
A_1^r \hookrightarrow A_1^2 \oplus D_{2m-2} \hookrightarrow A_1^2 \oplus D_{2m-1} \hookrightarrow D_{2m+1},
\end{eqnarray}
where the middle embedding is primitive (by what we have argued before)
and the primitive closure of $A_1^2 \oplus D_{2m-1}$ inside $D_{2m+1}$ is given, 
in terms of the previous embedding with image
$\langle d_1\rangle\oplus\langle\gamma_{2m+1}\rangle\oplus\langle d_3,\hdots,d_{2m+1}\rangle$,
by adjoining the root $\delta = (d_1+\gamma_{2m+1}+d_{2m}+d_{2m+1})/2$.
In turn,
 the primitive closure of $A_1^2\oplus D_{2m-2}=\langle d_1\rangle\oplus\langle\gamma_{2m+1}\rangle\oplus\langle d_4,\hdots,d_{2m+1}\rangle$  inside $D_{2m+1}$
amounts to adjoining $\delta$, too
(equivalently, one adjoins the root $d_2+d_3$ to  $A_1^2\oplus D_{2m-2}$).
Hence, the primitive closure of
$A_1^2\oplus D_{2m-2}$ inside $D_{2m+1}$ is 
isometric to $D_{2m}$, and the claim of the lemma follows in this case.

If the embedding  \eqref{eq:embed} does not involve the first orthogonal summand of the target lattice,
then along the same lines
$$A_1^{r-1}\hookrightarrow D_{2m-1} \;\; \Longrightarrow  \;\;
A_1^{r-1}\hookrightarrow D_{2m-2} \hookrightarrow D_{2m-1}.
$$
The chain of embeddings \eqref{eq:case1}
has to be modified to
\begin{eqnarray}
\label{eq:case2}
A_1^r \hookrightarrow A_1 + D_{2m-2}  \hookrightarrow A_1^2 + D_{2m-2} \hookrightarrow D_{2m+1},
\end{eqnarray}
but still the second rightmost lattice has primitive closure $D_{2m}$ inside $D_{2m+1}$ as claimed.

\end{proof}

With all Dynkin types covered, the proof of Lemma \ref{lem:roots} is complete.

\end{proof}

\subsection{Conclusion of the proof of Proposition \ref{prop:N_v}}
\label{ss:78}

Lemma \ref{lem:roots} furnishes us with an isotropic subgroup $H\subset M'/M$ of size $4$
whose non-zero elements are represented by roots $v_j\in M'\setminus M$.
Since $H\cong(\ZZ/2\ZZ)^2$ (because $H$ is a subgroup of $A_M\cong (\ZZ/2\ZZ)^r$), there is a relation
\begin{eqnarray}
\label{eq:mod_M}
v_1 + v_2 + v_3 = 0 \mod M.
\end{eqnarray}

For each $j$, we have $2v_j\in M$, so we can 
consider the unique pre-images of these vectors in $\tilde M$:
\[
w_j = 2v_j + m_j F \in \tilde M.
\]
Since $\pi|_{\tilde M}$ is injective, \eqref{eq:mod_M} implies that
\[
w_1 + w_2 + w_3 
\in 2\tilde M.
\]
Hence $(m_1+m_2+m_3)F\in 2\tilde M$,
and since no multiple of $F$ is in $\tilde M$ (because $\tilde M$ is negative-definite),
we infer that $m_1+m_2+m_3=0$.
In particular, one of the $m_j$ is even
and the corresponding vector $w_j\in\tilde M$ is 2-divisible in $\tilde V$.
Applying the inverse reflections $\sigma^{-1}$,
we obtain a linear combination $\sigma^{-1}(w_j)$ of the $C_i$
such that 
$$\sigma^{-1}(w_j)/2\in\Pic(X)\setminus \langle C_1,\hdots, C_r\rangle_\ZZ
$$
is a root.
This contradicts Theorem \ref{thm:nodal}
and thus completes the proof of Proposition \ref{prop:general}. 

\end{proof}

\subsection{Application to higher singularities}

\begin{prop}
\label{prop:A_3}
Let $X\subset \PP^3$ be quartic surface with 12 singular points
such that the minimal resolution $S$ is not supersingular.
Then all singularities have types $A_1$ or $A_2$.
\end{prop}

\begin{proof}
Assume there is a singular point of type $A_n\, (n\geq 3), D_k$ or $E_l$.
Then $S$ contains an $A_3$ configuration of smooth rational curves $C, C', C''$ lying above this singular point
and 11 smooth rational curves $C_1,\hdots,C_{11}$ lying above the other singularities.
Hence there is an isotropic vector
\[
E = H - (C+2C'+C'').
\]
Note that there are 13 disjoint smooth rational curves orthogonal to $E$, namely $C_1,\hdots, C_{11}, C, C''$.
Hence, Proposition \ref{prop:general} gives the contradiction that $S$ is supersingular.
\end{proof}

\begin{prop}
\label{prop:A_2}
Let $X\subset \PP^3$ be quartic surface with 12 singular points
such that the minimal resolution $S$ is not supersingular.
Then there are at most 3 singularities of type $A_2$.
\end{prop}

\begin{proof}
Assume that there are 4 singular points of type $A_2$ (or higher).
Denote the exceptional curves in $S$ lying above 3 of them by
\[
C_1, C_2, \;\; C_1', C_2', \;\; C_1'', C_2''.
\]
Consider the isotropic vector
\[
E = 3H + (C_1+2C_2) + (C_1'+2C_2') + 2(C_1''+2C_2'').
\]
Then $|E|$ induces an elliptic fibration on $S$,
and by the proof of Proposition \ref{prop:N_v},
the fibres contain orthogonal configurations of smooth rational curves of types
$A_2$ (lying above the fourth $A_2$ singularity) and
$A_1^{11}$ ($C_1, C_1', C_1''$ and exceptional curves above the remaining 8 singular points).
But  this is excluded by Proposition \ref{lem:12} (iii).
\end{proof}

\subsection{Proof of Theorem \ref{thm}}

By Theorem \ref{theo}, a normal quartic surface $X\subset\PP^3$ 
whose minimal resolution is not a supersingular K3 surface
contains at most 13 singular points.
We ruled out 13 nodes in \ref{ss:13nodes} and other configurations of rational double points in \ref{ss:vs}.
This proves that $X$ contains at most 12 singular points.
By Proposition \ref{prop:A_3}, they have types $A_1$ or $A_2$,
and by Proposition \ref{prop:A_2} there are at most 3 $A_2$'s.

The generic member of the family in Example \ref{ex:12} contains 12 nodes;
its minimal resolution is a K3 surface which is not  supersingular.
The family has dimension 22 in the space of quartics.
This completes the proof of Theorem \ref{thm}.
\qed

\section{General results}

Let us draw some consequences of Proposition \ref{prop:general}.
The first one is very much in line with Theorem \ref{theo}.

\begin{cor}
\label{cor2}
Let $X$ be a K3 surface in characteristic $2$.
Fix a positive vector $h\in\Pic(X)$ with square $h^2=2n>0$
and assume that $n$ is the sum of $r$ squares.
Then $X$ contains no more than $12+r$  $(-2)$-curves orthogonal to $h$
unless  $X$ is  supersingular.
\end{cor}


\begin{proof}
Assume that $X$ contains $m=13+r$ $(-2)$-curves $C_1,...,C_m$.
Writing $n=a_1^2+\hdots+a_r^2$,
we obtain the isotropic vector
\[
E = h - a_1 C_1 - \hdots - a_r C_r.
\]
This is perpendicular to the 13 $(-2)$-curves $C_{r+1},\hdots,C_m$,
so the genus one fibration induced by $|E|$ is quasi-elliptic
by Proposition \ref{prop:general}.
Equivalently, by \cite{rudakov-shafarevich}, $X$ is supersingular.
\end{proof}

In comparison with Theorem \ref{theo}, 
Corollary \ref{cor2} affords for much more flexibility as $h$ is not required to give a quasi-polarization;
also the $(-2)$-curves need not correspond to simple nodes
(e.g., we could also take two curves from an $A_3$-configuration or three curves from a $D_4$-configuration).

\subsection{Proof of Theorem \ref{theo2}}

As a second corollary, we indicate the proof of Theorem \ref{theo2}.
Assume that $X$ contains 13 $(-2)$-curves $C_1,\hdots,C_{13}$.
Their orthogonal complement in $\Pic(X)$ is a hyperbolic lattice of rank at least 5,
thus it represents zero by some non-trivial vector $E$ by Meyer's theorem  \cite[IV.3.2, cor.\ 2 to th.\ 8]{Serre:CA}).
By Proposition \ref{prop:general}, the genus one fibration induced by $|E|$
is quasi-elliptic.
But then $X$ is unirational, so $\rho(X)=22$, contradiction.
\qed

\begin{rem}
The result is sharp by virtue of the fibrations from \cite{shioda}.
\end{rem}

{
We can also prove that the 12 disjoint $(-2)$-curves in Theorem \ref{theo2} cannot be extended:

\begin{prop}
\label{prop:18-20}
Let $X$ be a K3 surface in characteristic $2$
with Picard number $18\leq \rho(X)\leq 20$.
If $X$ contains  12 disjoint ADE configurations,
then each has type $A_1$.
\end{prop}

\begin{proof}
Assume that at least one configuration has two components
and consider a subconfiguration of type $A_1^{11}+A_2$.
As in the proof of Theorem \ref{theo2}, its orthogonal complement 
has rank at least $5$, so the subconfiguration is supported on the singular fibres
of some elliptic fibration.
However,  this is impossible by Proposition \ref{lem:12} (iii).
\end{proof}
}

\begin{rem}
Note the following consequence of Proposition \ref{prop:18-20}: 
if $S$ in Theorem \ref{thm} has Picard number $18, 19$ or $20$,
then we can strengthen the statement to the extent that all 12 singular points of $X$ necessarily are nodes.
\end{rem}

\end{document}